\documentclass[12pt]{article}
	\usepackage{amsfonts}
	\usepackage{enumitem}
	\usepackage{mathtools}
	\usepackage{changepage}
	\usepackage{amsmath,amssymb}
	\usepackage{amsfonts}
	\usepackage{xcolor}
		\definecolor{pageColor}{rgb}{1,1,1}
		\definecolor{textColor}{rgb}{0,0,0}
		\pagecolor{pageColor} \color{textColor}
		\definecolor{sectionBox}{rgb}{0.6,0.4,0.8}
	\usepackage{relsize}
	\usepackage{xparse}
	\usepackage{titlesec}
	\usepackage{mathrsfs}
	\usepackage{romanbar}
	\usepackage{tikz-cd}
	\usepackage{framed}
	\usepackage{lipsum}
	\usepackage{graphicx}
	\usepackage{fancyhdr}
	\usepackage{subfig}
	\usepackage{booktabs}
	\usepackage{array}
	\usepackage{stackengine}
	\usepackage{wrapfig}
	\usepackage{scalerel}
	\usepackage{bm}
	\usepackage{stmaryrd}
	\usepackage{setspace}	
	\usepackage[b]{esvect}
	\usepackage{tikz}\usetikzlibrary{shapes.misc}
	\usepackage[parfill]{parskip}
	\usepackage[margin=1in]{geometry}
	\usepackage[framed,amsthm,amsmath,thmmarks,thref,hyperref]{ntheorem}
	\usepackage[hidelinks]{hyperref}
	\usepackage{eucal}
	\usepackage{colortbl}
	\usepackage{wasysym}
	\usepackage{pgfornament}
		\newcommand{\R}{\mathbb{R}}
		\newcommand{\Q}{\mathbb{Q}}
		
		\newcommand{\Z}{\mathbb{Z}}
		\newcommand{\N}{\mathbb{N}}

		\newcommand{\del}{\partial}

		\newcommand{\ST}{\mathcal{S_{T}}}
		\newcommand{\PT}{\mathcal{P_{T}}}
		\newcommand{\B}{\mathcal{B}}
		
		\newcommand{\LWT}{\mathcal{LW_{T}}}
		
		\newcommand{\ML}{\mathcal{ML}}
		\newcommand{\PML}{\mathcal{PML}}
		
		\newcommand{\T}{\mathcal{T}}
		\newcommand{\projsim}{\;\raisebox{-0.1cm}{$\stackrel{\mathrm{proj}}{\sim}$}\;}
		
		\DeclareMathOperator{\image}{image}

		\DeclareMathOperator{\dep}{dep}
		\DeclareMathOperator{\wt}{wt}
		\DeclareMathOperator{\cone}{cone^\circ}
		\DeclareMathOperator{\ccone}{cone}
		\DeclareMathOperator{\relint}{ri}
		\DeclareMathOperator{\aff}{aff}
		\DeclareMathOperator{\reldim}{reldim}
		
		\renewcommand{\emptyset}{\varnothing}
		\renewcommand{\le}{\leqslant}
		\renewcommand{\ge}{\geqslant}
		\renewcommand{\it}{\textit}

	\def\hmath$#1${\texorpdfstring{{\rmfamily\textit{#1}}}{#1}}

		\makeatletter
		\newlength\xvec@height%
		\newlength\xvec@depth%
		\newlength\xvec@width%
		\newcommand{\xvec}[2][]{%
		\ifmmode%
			\settoheight{\xvec@height}{$#2$}%
			\settodepth{\xvec@depth}{$#2$}%
			\settowidth{\xvec@width}{$#2$}%
		\else%
			\settoheight{\xvec@height}{#2}%
			\settodepth{\xvec@depth}{#2}%
			\settowidth{\xvec@width}{#2}%
		\fi%
		\def\xvec@arg{#1}%
		\def\xvec@dd{:}%
		\def\xvec@d{.}%
		\raisebox{.2ex}{\raisebox{\xvec@height}{\rlap{%
			\kern.05em
			\begin{tikzpicture}[scale=1]
				\pgfsetroundcap
				\draw (.05em,0)--(\xvec@width-.05em,0);
				\draw (\xvec@width-.05em,0)--(\xvec@width-.15em, .075em);
				\draw (\xvec@width-.05em,0)--(\xvec@width-.15em,-.075em);
				\ifx\xvec@arg\xvec@d%
				\fill(\xvec@width*.45,.5ex) circle (.5pt);%
				\else\ifx\xvec@arg\xvec@dd%
				\fill(\xvec@width*.30,.5ex) circle (.5pt);%
				\fill(\xvec@width*.65,.5ex) circle (.5pt);%
				\fi\fi%
			\end{tikzpicture}%
		}}}%
		#2%
		}
		\makeatother
		
		\renewcommand{\vec}[1]{\xvec[]{#1}}

		\makeatletter
		\let\save@mathaccent\mathaccent
		\newcommand*\if@single[3]{%
		\setbox0\hbox{${\mathaccent"0362{#1}}^H$}%
		\setbox2\hbox{${\mathaccent"0362{\kern0pt#1}}^H$}%
		\ifdim\ht0=\ht2 #3\else #2\fi
		}
		\newcommand*\rel@kern[1]{\kern#1\dimexpr\macc@kerna}
		\newcommand*\widebar[1]{\@ifnextchar^{{\wide@bar{#1}{0}}}{\wide@bar{#1}{1}}}
		\newcommand*\wide@bar[2]{\if@single{#1}{\wide@bar@{#1}{#2}{1}}{\wide@bar@{#1}{#2}{2}}}
		\newcommand*\wide@bar@[3]{%
		\begingroup
		\def\mathaccent##1##2{%
			\let\mathaccent\save@mathaccent
			\if#32 \let\macc@nucleus\first@char \fi
			\setbox\z@\hbox{$\macc@style{\macc@nucleus}_{}$}%
			\setbox\tw@\hbox{$\macc@style{\macc@nucleus}{}_{}$}%
			\dimen@\wd\tw@
			\advance\dimen@-\wd\z@
			\divide\dimen@ 3
			\@tempdima\wd\tw@
			\advance\@tempdima-\scriptspace
			\divide\@tempdima 10
			\advance\dimen@-\@tempdima
			\ifdim\dimen@>\z@ \dimen@0pt\fi
			\rel@kern{0.6}\kern-\dimen@
			\if#31
			\overline{\rel@kern{-0.6}\kern\dimen@\macc@nucleus\rel@kern{0.4}\kern\dimen@}%
			\advance\dimen@0.4\dimexpr\macc@kerna
			\let\final@kern#2%
			\ifdim\dimen@<\z@ \let\final@kern1\fi
			\if\final@kern1 \kern-\dimen@\fi
			\else
			\overline{\rel@kern{-0.6}\kern\dimen@#1}%
			\fi
		}%
		\macc@depth\@ne
		\let\math@bgroup\@empty \let\math@egroup\macc@set@skewchar
		\mathsurround\z@ \frozen@everymath{\mathgroup\macc@group\relax}%
		\macc@set@skewchar\relax
		\let\mathaccentV\macc@nested@a
		\if#31
			\macc@nested@a\relax111{#1}%
		\else
			\def\gobble@till@marker##1\endmarker{}%
			\futurelet\first@char\gobble@till@marker#1\endmarker
			\ifcat\noexpand\first@char A\else
			\def\first@char{}%
			\fi
			\macc@nested@a\relax111{\first@char}%
		\fi
		\endgroup
		}
		\makeatother
		\let\orisubsectionmark\subsectionmark
		\renewcommand\subsectionmark[1]{\label{subsec:#1}\orisubsectionmark{#1}}
	\newcommand{\pre}{\theoremprework{%
	\def\FrameCommand{{\color{sectionBox}{\vrule width 2pt\hspace{6pt}}}}}}
		\theoremframepreskip{0.3 cm}
		\theoremframepostskip{0.2 cm}
		\theoreminframepreskip{0.08 cm}
		\theoreminframepostskip{0.04 cm}

		\pre\newframedtheorem{thm}{Theorem}[section]
		\pre\newframedtheorem{claim}[thm]{Claim}
		\pre\newframedtheorem{lemma}[thm]{Lemma}
		\pre\newframedtheorem{corollary}[thm]{Corollary}

		\theoremstyle{nonumberdefinition}
		\pre\newframedtheorem{definition}{Definition.}
		\theoremprework{%
			\def\FrameCommand{{\color{sectionBox}{\vrule width 2pt\hspace{6pt}}}}
			}
			\theoremstyle{empty}
		\newframedtheorem{abstractNew}{}
		\theoremprework{%
			\def\FrameCommand{{\color{sectionBox}{\vrule width 2pt\hspace{6pt}}}}
			}
			\theoremstyle{empty}
		\newframedtheorem{reference}{}
	\title{\vspace{-3 cm}Growth Rate of Essential Surfaces\vspace{-0.2 cm}\\via Measured Laminations\vspace{-0.5 cm} \\\rule{0.3\textwidth}{.4pt}
	\vspace{-0.7 cm}} 

	\author{\small Brevan Ellefsen}
	\date{}
		\hyphenchar\font=-1
	\newcommand\titlebar{%
		\tikz[baseline,trim left=-0.1 cm - 2em,trim right = 0cm] {
			\fill [sectionBox!30!pageColor!100,rounded corners =0.25cm] (-0.1 cm - 2em,-1ex) rectangle (\textwidth - 0.1cm - 2em,2.5ex);
			\node [
				anchor= base east,
				minimum height=3.5ex] at (0cm,0) {
				\textbf{\S\thesection}
			};
		}%
	}

	\titleformat{\section}[block]{\Large\bfseries}{\titlebar}{0.1 cm}{}
	\renewcommand*{\thesection}{\arabic{section}}

	\newcolumntype{C}[1]{>{\centering\arraybackslash}p{#1}}
	\newlength{\tableWidth}
	\newcommand{\centered}[1]{\begin{tabular}{C{\the\tableWidth-25pt}}  #1 \vspace{0.1cm}\end{tabular}}
	\newcolumntype{?}{!{\vrule width 1pt}}

	\titleformat{\subsection}[runin]{\normalfont\bfseries}{\thesubsection}{1 ex}{\addperiod}
	\titleformat{\subsection}[runin]
	  {\normalfont\bfseries}
	  {\thesubsection}
	  {1ex}
	  {\addperiod}
	\newcommand{\addperiod}[1]{#1.}
	\titlespacing*{\section}{0pt}{4ex}{2.5ex}
	\graphicspath{ {./images/} }
    \hypersetup{nolinks=true}
	\hypersetup{final}
\begin{document}
\maketitle
\thispagestyle{empty}
\begin{abstractNew}[Abstract.]
We show that in any Haken 3-manifold $M$, the dimensions of $\ML(M)$ and $\ML_0(M)$ can be calculated using normal surfaces. As a corollary, we prove the number of closed orientable essential surfaces of Euler characteristic at least $-2n$ in a closed orientable hyperbolic 3-manifold with $H_2(M,\Z/2\Z) = 0$ is asymptotic to $n^{\dim \ML_0(M)}$.
\end{abstractNew}

\tableofcontents

\section{Introduction}
The \it{measured lamination space} of a surface $S$, also written $\ML(S)$, is extremely useful in the theory of surfaces, being equivalent to the Teichmuller space of $S$. Motivated by this, one wishes to see what holds for $3$-manifolds. In this paper we shall adhere to the definition of $\ML$ as put forth by Oertel in \cite{oertelMeasuredLaminations3Manifolds1988}:

\begin{definition}[$\ML$]
	The \it{measured lamination space} of a 3-manifold $M$ is (as a set) the collection of all measured incompressible prelaminations carried by TIBs, up to equivalence by pinching, splitting, and isotopy.
\end{definition}

Oertel then topologized $\ML$ by first projectivizing to get $\PML$, then finding a bijection from $\PML$ into the space of length functions, then pulling back the induced topology. In this paper, we will use this topology on $\ML$. Importantly, Oertel shows this topology agrees with the one induced locally from cones over branched surfaces (each with their induced topology from $\R^n$).

Understanding the structure and construction of $\ML$ remains a significant open problem. Nevertheless, this does allow one to define the dimension of $\ML$ in terms of branched surfaces. Since the components of $\ML$ contain the images of the (open) cones of weights over certain branched surfaces $\{\B_i\}$ and are contained within their closures (see \cite{oertelMeasuredLaminations3Manifolds1988}), we can define $\dim \ML$ via $\dim \ML \coloneqq \max_i \dim \cone(\B_i)$.

Although Hatcher and Oertel were able to formalize a great deal about $\ML$, the difficulty in its construction and lack of apparent applications resulted in the space remaining dormant for many years. Interest was recently renewed by \cite{dunfieldCountingEssentialSurfaces2022}, in which the number of essential surfaces in certain $3$-manifolds were shown to have polynomial growth. In that paper, the following conjecture relating to the power of the growth-rate polynomial was proposed, where $\ML_0$ denotes the subset of $\ML$ disjoint from $\del M$, $C$ is a face of the projective normal complex associated to a fixed triangulation of $M$, and $W_C$ is a vector space associated to $C$:
{\hyphenpenalty=10000
\begin{reference}{(Conjecture 1.12 of \cite{dunfieldCountingEssentialSurfaces2022})}
Let $M$ be a compact orientable irreducible $\del$-irreducible atoroidal acylindrical 3-manifold that does not contain a closed nonorientable essential surface. The dimension of $\ML_0$ is the maximum of $\dim C - \dim W_C + 1$ among essential lw-faces $C$.
\end{reference}
}
In this paper we make progress toward understanding the measured lamination space by resolving this conjecture via the following, more general result:
{\hyphenpenalty=10000
\begin{reference}[Theorem \ref{thm:mainResult}.]
Let $M$ be a compact orientable irreducible $\del$-irreducible 3-manifold. Then:
	{
	\begin{enumerate}[itemsep=0cm,leftmargin=2em,font=\upshape]
		\item If $\tau$ is a material triangulation of $M$, then $\dim( \ML)$ is the \\maximum of $\dim C - \dim W_C + 1$ among complete lw-faces $C$.
		\item If $\tau$ is an ideal triangulation of $M$, then $\dim( \ML_0)$ is the \\maximum of $\dim C - \dim W_C + 1$ among essential lw-faces $C$.
	\end{enumerate}}
\end{reference}
}

These conditions on $M$ are natural; indeed, in the case $\ML$ is nontrivial the conditions are essentially just that $M$ is Haken (under the convention that such spaces are assumed to be $\del$-irreducible): indeed, if $M$ admits an essential measured lamination then it admits a rational one by density and thus an integral one by scaling, which can be identified with an essential surface. Doubling weights as necessary ensures this surface is orientable.

As shown in \cite{dunfieldCountingEssentialSurfaces2022}, the resolution of their conjecture yields the asymptotic growth rate on the count of essential surfaces of Euler characteristic at least some bound. More precisely, we get the following:

\begin{corollary}
	Let $M$ be a compact orientable irreducible $\del$-irreducible atoroidal acylindrical 3-manifold that does not contain a closed nonorientable essential surface. The number of closed orientable essential surfaces in $M$ with Euler characteristic at least $-2n$ is asymptotic to $c\, n^{\ML_0(M)}$ for some constant $c$.
\end{corollary}

\section*{Acknowledgements}
We thank Lara Lahey and Saaber Pourmotabbed for their feedback. We especially thank Nathan Dunfield for the many suggestions he provided. Ellefsen was partially supported by U.S. National Science Foundation grant DMS 2303572. This material is based upon work supported by the National Science Foundation under Grant No. DMS-1928930, while the author was in residence at the Simons Laufer Mathematical Sciences Institute in Berkeley, California, during Spring semester of 2026.
\pagebreak
\section{Background and Conventions}

We begin by recalling results we will need from various papers, and develop necessary background to understand the constructions to follow. Throughout the following, $M$ will denote a compact, irreducible, $\del$-irreducible, and orientable $3$-manifold, possibly without boundary, with a fixed triangulation $\T$. \it{(By `triangulation', we mean collection of tetrahedra glued pairwise along their faces.)} All surfaces (barring compression disks) will be properly embedded.

We begin by fixing some terminology. A surface is \it{incompressible} when it admits no compression disks and is neither a sphere nor disk, and a surface is \it{$\del$-incompressible} when it admits no $\del$-compression disks and is not itself a disk. A surface is \it{injective} (resp. \it{$\del$-injective}) iff a regular neighborhood of it is incompressible (resp. $\del$-incompressible). Note an orientable surface is injective iff it is incompressible, and is likewise $\del$-injective iff it is $\del$-incompressible. An \it{essential surface} is an orientable surface whose components are each injective, $\del$-injective, and not $\del$-parallel. The union of $n$ disjoint copies of an orientable surface $F$ is called a \it{multiple} of $F$ and is denoted by $nF$. For non-orientable surfaces, $nF$ is defined as either $\tfrac{n}{2}G$ or $F \cup \tfrac{n-1}{2} G$ depending on the parity of $n$, where $G$ is the boundary of a regular neighborhood of $F$. We say two orientable surfaces are \it{projectively isotopic} if some multiples of the surfaces are isotopic. We will denote isotopy by $\sim$ and projective isotopy by \!$\projsim$.

We will denote the relative interior and relative boundary of a convex set $X$ by $\relint X$ and $\del X$ respectively.

\subsection{Branched Surfaces}

The first major tool is that of a branched surface, which is to a $3$-manifold what a train track is to a $2$-manifold. Accessible treatments can be found in \cite{floydIncompressibleSurfacesBranched1984,oertelIncompressibleBranchedSurfaces1984,gabaiEssentialLaminations3Manifolds1989}, though we will here recount the necessary basics. A branched surface is modeled locally as a subset of Figure \ref{fig:branchedSurface}. The dotted lines signify the singular locus, which is $1$-dimensional except at triple points. At each point on a branched surface is a well-defined tangent space.

\begin{wrapfigure}{R}{0.45\textwidth}
  \begin{center}
    \includegraphics[width=0.4\textwidth]{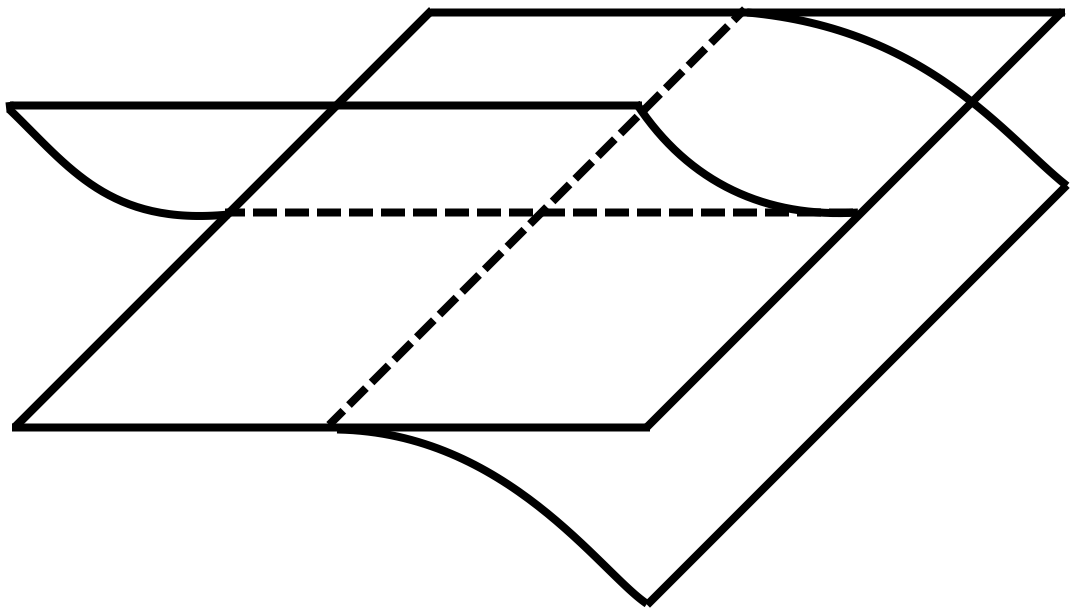}
  \end{center}
  \caption{}
  \label{fig:branchedSurface}
\end{wrapfigure}

Given a branched surface $B$, we can thicken each sheet (as well as the singular locus) of $B$ vertically to get a fibered neighborhood $N = N(B)$. The \it{sectors} of $B$ are the closures of the components of $B$ minus its singular locus. There is a clear projection map $\pi$ from $N$ to $B$ given by collapsing fibers of $N$ (so that $\pi(N) = B)$.  The \it{horizontal boundary} $\del_h N$ is the union of the endpoints of the I-fibers. The \it{vertical boundary} $\del_v N$ is the closure of the complement of $\del M \cup \del_h N$ in $\del N$; geometrically, it is the portion of the interiors of the I-fibers intersecting $\del N$. To each sector $B^i$ of $B$ we may assign a weight $w_i \ge 0$. We say an associated vector $w = (w_i) \in \R^s$ (where $s$ is the number of sectors) is an \it{invariant measure} or \it{weight system} provided the weights satisfy the branch equations $w_1 = w_2 + w_3$ wherever two sheets merge into one. Given a weight system on $B$ with integer coordinates, we can construct an associated (non-branched) surface $B(w)$ \it{carried} by $B$ by gluing $w_i$ many parallel copies of each sector $B^i$ in the natural way. We can always scale a rational solution to be integral, so we will often identify rational points with their surface lifts.

Of special importance to us are \it{incompressible} branched surfaces. We say a branched surface $B$ is incompressible given the following:
\begin{enumerate}[leftmargin=1.5cm]
	\item $B$ has no disks nor half-disks of contact.
	\item $N_h$ is incompressible and $\del$-incompressible in $M - N^\circ$.
	\item There are no monogons in $M - N^\circ$.
\end{enumerate}

Recall a surface $F$ is called injective (resp.~$\del$-injective) iff a regular neighborhood of $F$ is incompressible (resp.~$\del$-incompressible). Incompressible branched surfaces carry only injective and $\del$-injective surfaces, and the converse also turns out to be true. Indeed, in Theorem 3 of \cite{oertelIncompressibleBranchedSurfaces1984} it is shown that there exists a finite collection of incompressible branched surfaces in $M$ such that every orientable, incompressible, and $\del$-incompressible surface in $M$ is carried with positive weights by a branched surface of the collection.

A branched surface carrying some surface with strictly positive weights is called \it{recurrent}. If $B(w)$ with $w_i > 0$ is isotopic to $B(w')$ with $w' \ge 0$ iff $w = w'$, we say $B$ has no \it{isotopy relations}. A closed curve transverse to $B$ is called a \it{transversal}. A curve or arc $\gamma$ which is transverse to $B$ is said to be \it{efficient}  if we cannot simplify its intersection with $N$ in the following sense: no arc of $\pi^{-1}(\gamma)$ can be pushed into $\partial_h N$ while keeping endpoints fixed; formally, we require no arc of $\pi^{-1}(\gamma) - N^\circ$ to be homotopic in $M - N^\circ$ rel endpoints to an arc in $\partial_h N$. A branched surface is \it{transversely recurrent} if every point of $B$ admits an efficient transversal.

We say a branched surface is \it{birecurrent} if it is both recurrent and transversely recurrent. A branched surface which is both birecurrent and incompressible is called a \it{TIB}. A TIB carries no $\del$-parallel surface: indeed, an efficient transversal for the TIB would yield an efficient transversal for the $\del$-parallel surface, which can not occur. The most important classical result concerning TIBs is the following:

\begin{reference}[Theorem 4.1 of \cite{oertelMeasuredLaminations3Manifolds1988}.]
	Given $M$ orientable, irreducible, and $\del$-irreducible, there is a finite collection of TIBs without isotopy relations such that every orientable essential surface in $M$ is carried with positive weights by a branched surface of the collection.
\end{reference}

We would be remiss to not mention the unfinished notes \cite{hatcherMeasuredLaminationSpaces1999} of Hatcher. In them, Hatcher proposed a slightly different definition of $\ML$ endowing a PL-structure via charts, but potentially with some frontier strata missing. In particular, this means $\ML$ may be incomplete, and correspondingly its projectivization $\PML$ may be non-compact. It is generally understood that Hatcher's construction is equivalent to Oertel's, though we stick with the latter in these notes. Note that in both definitions, every essential surface can be realized in multiple ways in $\ML$ via thickening, so isolated compact leaves are excluded to prevent ambiguity. However, our arguments in this paper will mostly work purely at the level of cones, where we can think of integral (and therefore rational) points as surfaces and work with them as such; as such, we do not concern ourselves with prelaminations here.

\subsection{Normal Surfaces}

The idea of a normal surface is to chop a given surface into pieces so each $3$-simplex $\Delta \in \T$ sees only triangular or quadrilateral pieces of the surface. To be precise, we first define a \it{normal arc} to be a properly embedded arc on a $2$-face $[f]$ of $\Delta$ whose endpoints lie on distinct edges of $[f]$. A \it{normal curve} is then a simple closed curve whose intersection with any $2$-face $[f]$ of $\Delta$ is a union of normal arcs. Since $\Delta$ inherits a geometry from $M$, it suffices to assume all normal arcs are straight line segments. It is easy to prove a normal curve which meets each edge of $\Delta$ at most once has length $3$ or $4$.

We then define {normal tris} and {normal quads} to be properly embedded disks in $\Delta$ whose boundaries are normal curves of length $3$ and $4$ respectively. A \it{normal disk} is a normal tri or normal quad. Separating normal disks into equivalence classes based on how they separate the vertices of $\Delta$, it is easy to see there are four types of tris and three types of quads. A \it{normal surface} is a surface $F$ in $M$ whose intersection with any $\Delta$ is comprised of disjoint normal disks. A \it{normal isotopy} is an isotopy through normal surfaces. Equivalently, a normal isotopy is an isotopy which preserves the normal disk types and counts in each $\Delta$. As was shown by Haken, every incompressible surface is isotopic to a normal surface. For a proof and more details, see \cite{schultensIntroduction3manifolds2014}.

\subsection{Normal Solution Spaces}

One might hope to present a normal surface by specifying the number of instances of each disk type in each $3$-simplex. Since there are exactly $7$ normal disk types, this would be equivalent to specifying a vector in $\N^{7t}$, where $t = |\T|$ (note a disk type might not appear, in which case its coordinate is zero). This idea works, but is rigid and does not tell us which lattice points correspond to normal surfaces. To allow flexibility we extend to $\R^{7t}$, and to specify which lattice points are normal surfaces we invoke the \it{matching equations}  $x_i + x_j = x_k + x_l$, which specify how normal disks in adjacent tetrahedra meet along each possible arc type.

We then define the \it{normal solution space} $\ST$ to be the subspace of $\R^{7t}$ satisfying the matching equations. One can prove that each normal surface $F$ in $M$ embeds into this space as $\vec F \in \N^{7t}$; in fact, the normal surfaces are precisely the lattice points in $\ST$ which are \it{admissible}, meaning there is at most one nonzero quad coordinate in each tetrahedron. Note $\ST$ is a finite-sided polyhedral cone. Contained within $\ST$, one can define various notions of a \it{projective solution space} by normalizing according to some surface invariant. For our purposes here, it will be convenient to normalize by \it{weight}. For a normal surface $F$, this is given by the number of intersections of $F$ with the $1$-skeleton of $\T$. For a general $x \in \R^{7t}$, this is extended linearly. A preliminary definition would be to merely sum over the components of $x$, adding $1$ for each vertex of the associated elementary disk (so tris would contribute $3$, and quads $4$). This would overcount though, since a normal surface (without boundary) which intersects an edge must also intersect all tetrahedrons adjacent to said edge, resulting in overcounting the intersection precisely by the valency of the edge. To fix this, given an elementary disk $E_i$, we define $c_i$ to be the sum of the reciprocals of the valencies of the edges of $\T$ which the vertices of $E_i$ lie upon. We then define $\wt(x) = \sum_i c_i x_i$, yielding $\wt(F) = \wt(\vec F)$ for every normal surface $F$.

We now define the projective solution space $\PT$ to be the normalization of $\ST$ by weight, i.e. $\PT = \{x \in \ST \mid \wt(x) = 1\}$.  We say a surface $F$ is \it{carried} by a subset $X$ of $\PT$ if the projectivization $F^* \coloneqq \tfrac{1}{\wt{F}}\vec F$ is in $X$ (likewise, $X$ is then said to \it{carry} $F$). A face $C$ of $\PT$ is said to be the \it{carrier} of $F$ if $F$ is carried by $C^\circ$. Letting the interior of a vertex be itself, every normal surface is carried by the interior of some unique face of $\PT$.

Given a set $A \subseteq \R^{7t}$, we define the result of \it{collapsing isotopy relations} to be the quotient of $A$ by the subspace generated by lines between admissible lattice points in $A$ representing isotopic normal surfaces. 

\subsection{Least Weight Surfaces}

Since isotopy classes are a primary object throughout the discussions here, it makes more sense to consider canonical representatives of the classes. We achieve this via \it{least weight surfaces}, often written as \it{lw-surfaces}. These are defined to be compact orientable incompressible $\del$-incompressible normal surfaces that are least weight among all normal surfaces in their isotopy classes. The relevance of orientability here is to ensure the double of an lw-surface remains least weight in its isotopy class; this is not always true for non-orientable surfaces of least weight. Also note that a given surface may have finitely many distinct lw-surfaces in its isotopy class.

We define an \it{lw-face} to be a face of $\PT$ such that every surface carried by it is an lw-surface. We further define a \it{complete face} to be an lw-face which, whenever it carries an orientable normal surface $F$, carries every lw-surface isotopic to $F$.

The collection of all lw-faces forms a complex $\LWT$. An important result of \cite{tollefsonIsotopyClassesIncompressible1995} is the following:

\begin{reference}[(\cite{tollefsonIsotopyClassesIncompressible1995}, Theorem 4.5)]
	Every lw-surface is carried by a complete lw-face. In particular, any lw-face is contained in some complete lw-face. 
\end{reference}

As in \cite{oertelIncompressibleBranchedSurfaces1984, tollefsonIsotopyClassesIncompressible1995}, given any lw-face $C$ serving as the carrier of an lw-surface $S$, we can build a branched surface $\tilde B_S$ by flattening and identifying elementary disks in $2S$ (the boundary of a regular neighborhood of $S$; we do not actually need $S$ to be orientable for this to work). In the process we may create `thin sectors', which can be made to lie within a small regular neighborhood of $\T^2$. The branched surface $\tilde B_S$ may not be incompressible, so we remove disks of contact to get an incompressible branched surface $B_S$. While the construction of $\tilde B_S$ means the surfaces it carries are precisely the normal surfaces carried by $C$, the same is not necessarily true of $B$. While everything carried by $B$ is still carried by $C$, the converse is not generally true: the compression process can result in $B$ carrying strictly fewer surfaces than $\tilde B$ (up to normal isotopy). However, we \it{do} get a converse up to (potentially non-normal) isotopy, given as Lemma 3.1 in Tollefson's paper:

\begin{reference}[(\cite{tollefsonIsotopyClassesIncompressible1995})]
	Every normal surface carried by $C$ is isotopic to a surface carried by $B$.
\end{reference}

We will later make use of this fact to relate $\ccone(C)$ and $\ccone(B)$.

\subsection{PIC-Partitions}

In \cite{tollefsonIsotopyClassesIncompressible1995}, Tollefson defined the concept of a \it{PIC-partition} (PIC standing for ``Projective Isotopy Class"). The PIC-partition of a face $C$ of $\LWT$ is a foliation of $C^\circ$ along affine subspaces such that two surfaces carried by $C$ are projectively isotopic iff their projectivizations lie along the same leaf of the foliation. As detailed in \cite{dunfieldCountingEssentialSurfaces2022}, by choosing a different projective realization of $\PT$ than Tollefson used (normalizing by weight rather than $L^1$ norm), the PIC-partition will be chosen to be along \it{parallel} affine subspaces. 

In short, the idea is as follows: first, consider a normal surface $F$ with carrier $C$. Let $V_F \subseteq \R^{7t}$ be the space generated by all $\vec G$ with $G \projsim F$, which is finite dimensional since $\R^{7t}$ is. We projectivize $V_F$ by weight to get $X_F$, which is equivalent to intersecting $V_F$ with $\PT$. $X_F$ is affine, being the intersection of affine spaces, so has a well-defined tangent (vector) space $TX_F$ given by the subtraction of elements in $X_F$. One can show this choice is independent of the surface $F$ chosen, and we let $W_C \coloneqq TX_F$. In general, the tangent space functor taking a space to its tangent space at the origin will be denoted by $T$ throughout this paper, and is explicitly given by $TX \coloneqq \{x_1 - x_2 \mid x_1, x_2 \in X\}$.

This refined notion of the PIC-partition of a face $C$, with leaves being translates of a vector space the authors labeled $W_C$, will be the convention followed here. It is important to emphasize that $W_C$ is a specific vector space and not the foliation itself; however, because the foliation is via parallel copies of $W_C$, the foliation will often be identified with $W_C$. 

One important aspect of $W_C$ is how it interacts with $\dep(C)$. We say a face $D$ of an lw-face $C$ is \it{$C$-independent} if no normal surface carried by $D$ is projectively isotopic to a normal surface carried by $C^\circ$. We then define $\dep(C)$ as the complement in $C$ of the set of $C$-independent faces. Note $\dep(C)$ is likewise the union of the interiors of the $C$-dependent faces (note the interiors are all disjoint, so this is a partition). As proven in \cite{dunfieldCountingEssentialSurfaces2022}, $\dep(C) = \{ x \in C \mid x + W_C \text{ meets } C^\circ\}$, so that $\dep(C)$ is the largest subset of $C$ we get by `extending' the affine pieces in the foliation of $C^\circ$ by $W_C$ in a nonsingular way to the boundary. These lines then represent projective isotopy classes in the general case, and isotopy classes in the orientable case. For more precise statements, see \cite{dunfieldCountingEssentialSurfaces2022}.

\subsection{Cones}

The concept of cones will be quite important throughout this paper, so we briefly mention it. Given a face $C$ of $\PT$, we form the cone over $C$ in $\R^{7t}$, which we denote $\ccone(C)$. We may likewise consider the interior of this cone, which we denote $\cone(C)$ (note the origin is not included). Note a surface $F$ is carried by $C$ iff $\vec F \in \ccone(C)$, and has carrier $C$ iff $\vec F \in \cone(C)$; in other words, the lattice points in the cone over $C$ exactly represent surfaces carried by $C$.

We get an analogous story in the case of branched surfaces. Given a branched surface $B$, we can consider the set all of non-negative integer weights systems in $\R^s$ over $B$ (meaning each weight is non-negative), which we denote $\ccone(B)$. We may likewise consider the set all of \it{positive} integer weights systems in $\R^s$ over $B$ (meaning each weight is positive), which we denote $\cone(B)$.

In general, given a closed subset $A$ of $\R^n \setminus \{0\}$, the \it{cone} through $A$ will be defined as $\R_{\ge 0} A$, i.e. $\{ra \mid r \in \mathbb \R_{\ge 0}, a \in A\}$. This will be denoted by $\ccone(A)$. We also often consider the interior of the cone, given by $\cone(A) \coloneqq \R_{>0} A^\circ$.

We now recall the construction of the map $\rho$ from \cite{tollefsonIsotopyClassesIncompressible1995}. Letting $\B = \B_S$, we will denote by $\B^i$ the $i$th sector of $\B$, and will let $\xvec{\B^i}$ denote the normal coordinate representation of $\B^i$ coming from the normal disks used to build the sector (if $\B^i$ is a thin sector and thus has no associated normal disks, we define $\B^i$ as the zero vector). For any sector $\B^i$ (thin or otherwise), we associate it with the weight vector $(0, \ldots, 1, \ldots, 0)$ (with a $1$ in the $i$th position and zeroes elsewhere). This allows us to write any weight vector $w = (w_i)$ as $\sum_i w_i \B^i$, so that $\rho \colon \R^s \to \R^{7t}$ is simply the linear map $\rho\left ( \sum_i w_i \B^i\right) = \sum_i w_i \rho(\B_i) = \sum_i w_i \vec \B_i$ sending the set $\{\B^i\}$ to the set $\{\vec{\B^i}\}$.

Note that the definition of $\rho$ immediately yields $\rho(\vec f) = \vec F$; indeed, the weights on thin sectors are uniquely determined by the non-thin sectors (the `\it{thick}' sectors), and forgetting the thin sectors yields (up to normal isotopy) a union of sectors formed from normal disks which were originally glued to make $S$.

\section{Minimal Completions}

We will say a face $B$ is \it{compatible} with $C$ if $W_B = W_C \cap TB$. We will sometimes instead say $B$ is \it{$C$-compatible} or \it{satisfies the compatibility condition}. Note it is always true that $W_B \subseteq W_C \cap TB$; compatible faces are special in that we get equality, meaning their foliations have maximal possible dimension. If $B$ is $C$-independent then $B$ is $C$-compatible:

\begin{lemma}
Let $C$ be a lw-face and $B$ be a $C$-dependent face. Then:
	\begin{enumerate}[itemsep=0cm]
		\item $\dep(B) \subseteq \dep(C)$.
		\item  $B$ is compatible with $C$.
	\end{enumerate}
\end{lemma}
\begin{proof}
	Let $x \in \dep(B)$, so there exists some $w \in W_B \subseteq W_C$ such that $x + w = b \in B^\circ$. Since $B^\circ \in \dep(C)$ (using that $B$ is $C$-dependent), there exists some $w' \in W_C$ such that $b + w' \in C^\circ$, so $x + (w + w') \in C^\circ$. Since $w + w' \in W_C$, we conclude $x \in \dep(C)$.

	For the second statement, we must prove $W_C \cap TB \subseteq W_B$. Let $w \in W_C \cap TB$. Since $w \in TB$, there exist $b_1, b_2 \in B$ such that $w = b_1 - b_2$. Since faces have rational vertices, we can assume both $b_1$ and $b_2$ are rational. Likewise, we can assume both are contained in the interior (if one is not, replace it by a rational on the line segment between them). By scaling, we find vectors $\vec F$ and $\vec G$ associated to surfaces $F$ and $G$ carried positively by $B$ such that $F^* = b_1$ and $G^* = b_2$. Since $F^* - G^* = b_1 - b_2 = w \in W_C$ with $F$ and $G$ carried by $B^\circ \subseteq \dep(C)$, we know by Theorem 3.5 of \cite{dunfieldCountingEssentialSurfaces2022} that $F \projsim G$. Turning this on its head, we can apply the same theorem with the observations $F^*, G^* \in B^\circ \subseteq \dep(B)$ and $F \projsim G$ to conclude $F^* - G^* \in W_B$.
\end{proof}
{\setlength{\parskip}{0.2em}
Recall from \cite{dunfieldCountingEssentialSurfaces2022} that every essential surface is carried by $\dep(C)$ for some unique, complete face $C$. Also note $\dep(C)$ is a union of open lw-faces, and every lw-face carries an lw-surface in its interior by density of rationals. Carrying a surface extends to carrying the whole open lw-face carrying $S$ positively, so we can define the \it{minimal completion} of a lw-face:

\begin{definition}[\it{minimal completion}]
	Given an lw-face $C$, we define the \it{minimal completion} $\widebar C$ as the unique complete lw-face $\widebar C$ such that $C^\circ \in \dep(\widebar C)$.
\end{definition}

Whenever the context is clear, we shall refer to $\widebar C$ simply as the \it{completion} of $C$. Another notion in this paper we will need is that of the \it{relative dimension} of a lw-face $C$, defined as $\reldim C \coloneqq \dim C - \dim W_C$. The motivation for such a definition comes from overcounting: directions in $W_C$ correspond to isotopy relations, so by subtracting the dimension of $W_C$ we remove overcounting in lattice-counting estimates regarding $C$. We now prove a foundational result on relative dimension:

\begin{lemma}
	The relative dimension of an lw-face $C$ is less than or equal to the relative dimension of its completion.
\end{lemma}
\begin{proof}
	We must show $\dim C - \dim W_C \le \dim \widebar C - \dim W_{\widebar C}$. Since $C^\circ \in \dep(\widebar C)$, we know $W_C = W_{\widebar C} \cap TC$, so $\dim W_C = \dim {W_{\widebar C}} + \dim {TC} - \dim(W_{\widebar C} + TC)$. Since $\dim TC = \dim C$, it suffices to prove $\dim(W_{\widebar C} + TC) \le \dim \widebar C$, which follows from $W_{\widebar C} + TC \subseteq T \widebar C$ since both $W_{\widebar C} \subseteq T\widebar C$ and $TC \subseteq T \widebar C$.
\end{proof}

\begin{corollary}
	Relative dimension is maximized on some complete lw-face.
\end{corollary}

Also note that if $\dim(W_{\widebar C} + TC) = \dim \widebar C$ then we get equality of relative dimensions. In general, given any affine space $W$, a face $B$ of $C$ is said to be \it{$C$-filling} (with respect to $W$) if $TC = TB + W$. When the context is clear, we will often just call such a face \it{filling}.

\begin{corollary} \label{cor:reldimMaxOnComplete}
	$\reldim C = \reldim \widebar C$ iff $C$ is filling.
\end{corollary}

\section{Pullback Constructions}
We now aim to show that for any lw-surface $S$ carried positively by a face $C$ such that $\B_S$ is without isotopy relations, we have $\cone(\B_S) \cong \cone(C)/W_C$. The isomorphism is induced from $\rho$, and while one could slightly shorten the following arguments to directly prove this, doing so somewhat obfuscates the underlying projectivization process. Instead, we will explicate the projectivization and prove a more general result not requiring $\B_S$ to be without isotopy relations. Indeed, for any face $C$ and positively carried lw-surface $S$, we will define a space $W_{\B_S}$ such that $\cone(\B_S)/W_{\B_S} \cong \cone(C)/W_C$, with the isomorphism induced by $\rho$. In the case where $\B_S$ is without isotopy relations we will have $W_{\B_S} = 0$, proving the desired result. In the process we shall need a couple technical results from linear algebra, which have easy proofs and have been put in the appendix.
}
As a first step, we prove $\rho$ is injective. One important note here is that we will only allow $\B_S$ to vary by normal isotopy. This allows us to directly connect the geometry of surfaces carried by $\B_S$ with normal surfaces carried by $C$, since surfaces will geometrically be the same when carried by either $\B_S$ or $C$ (per the construction of $\B_S$), and not just isotopic. Also note $\rho$ maps each integral point in $\ccone(\B_S)$ to $\ccone(C)$ by construction, and thus does the same to rational points by scaling. Taking closures, and using the fact $\ccone(C)$ is closed, we find $\rho(\ccone(\B_S)) \subseteq \ccone(C)$, allowing us to simultaneously restrict and corestrict the map as needed. One might wish to likewise show $\rho(\cone(\B_S)) \subseteq \cone(C)$, but this is more nuanced; it could be, for example, that $\rho$ embeds $\ccone(\B_S)$ as a face of $\cone(C)$.

\begin{lemma}
	$\rho \colon \ccone(\B_S) \to \R^{7t}$ is injective
\end{lemma}
\begin{proof}
	If $\rho(\vec f) = \rho(\vec g)$ for some integral weights $\vec f$ and $\vec g$ associated to surfaces $F$ and $G$, then $\vec F = \vec G$ so $F$ and $G$ are normally isotopic. Since $\B_S$ is fixed up to normal isotopy, this means $\vec f = \vec g$; for the thick sectors this is determined by requiring the normal disk counts to agree to yield a normal isotopy, and the thin sector weights are determined from the thick sector weights by the switch equations. By scaling we get injectivity on rational points, and thus injectivity on the whole cone by \thref{claim:injectivityResult} using the ambient space $\aff(\ccone(\B_S))$.
\end{proof}

Because $\rho$ is injective on $\ccone(\B_S)$, we can freely pull back constructions over $C$ to constructions over $\ccone(\B_S)$. To simplify notation in what is to follow, we will often suppress the dependence of $\B_S$ on $S$ and just write $\B$. At the level of isotopy this is fine since each such $\B$ carries the same surfaces by Lemma 3.1 of \cite{tollefsonIsotopyClassesIncompressible1995}, but at the level of normal isotopy this is a slight abuse of notation for clarity of presentation.

While some of the definitions and arguments to follow in this paper would be slightly more naturally stated by restricting $\rho$ to the affine hull of the cone of $C_\B$, this has the tradeoff of disallowing arguments involving the origin (e.g. that $\rho$ pulls back vector spaces to vector spaces, or that $\rho$ sends $0$ to $0$). As such, we have opted to consider $\rho$ as a map of ambient spaces, but the reader is safe to restrict $\rho$ whenever it is natural to do so.

The first thing we must define is a notion of weight. For any $w \in \ccone(\B)$, we define the \it{weight} of $w$, written $\wt(w)$, via $\wt(w) \coloneqq \wt(\rho(w))$. This notion of weight is linear, being the composition of linear maps. We will often want to normalize a weight system $w$, yielding $w/\wt(w)$, which we denote by $w^*$. Since we will be projectivizing by weight in what follows, this will often be referred to as the projectivization of $w$. In the case a weight system is written $\vec f$, we will simply write $f^*$ to denote the normalization. One can check immediately from the definition of weight and the linearity of $\rho$ that $\wt(\vec f) = \wt(\vec F) = \wt F$, meaning weight for surfaces remains the intersection count with the $1$-skeleton. Also immediate is the formula $\rho(w^*) = (\rho(w))^*$, which finds occasional use in calculations.

Next we pull back $C$ itself, yielding $C_{\B} \coloneqq \rho^{-1}(C) \cap \ccone(\B)$. The map $\rho \colon C_{\B} \to C$ is thus well-defined, so extends via linearity to a map $\rho \colon \ccone(C_{\B}) \to \ccone(C)$.
Recalling every vector in $C$ has weight $1$, the following is immediate:

\begin{claim}\label{claim:ccone(B)=ccone(C_B)}
	$C_{\B} = \{w \in \ccone(\B) \mid \wt(w) = 1\}$
\end{claim}

\begin{proof}
	If $w \in C_{\B}$ then $\rho(w) \in C$ so $\wt(w) = \wt(\rho(w)) = 1$. Conversely, if $w \in \ccone(\B)$ has weight $1$ then $\rho(w) \in \ccone(C)$ has weight $1$ so lies in $C$.
\end{proof}

Note this means $C_{\B}$ is a convex compact polytope which does not intersect the origin. In particular, this tells us $\ccone(\B)$ is the cone over $C_{\B}$.

\begin{corollary}\label{cor:ccone(B)=ccone(C_B)}
	$\ccone(\B) = \ccone(C_{\B})$
\end{corollary}

Because of this, we will often use the notations $\ccone(C_{\B})$ and $\cone(C_{\B})$ in place of $\ccone(\B)$ and $\cone(\B)$ to emphasize the symmetries with the normal surface world.

The final new definition we make is that of $W_{\B}$. Analogously to $C_{\B}$, we define $W_{\B} \coloneqq \rho^{-1}(W_C) \cap TC_\B$. Since $\rho$ is linear, $W_{\B}$ is itself a vector space. Since $\rho$ is injective on $C_\B$ it is also injective on $TC_\B$, so we may identify $\rho(W_\B)$ with a subspace of $W_C$. Note the restriction of $\rho$ to $\aff(\ccone(C_\B)$ sends integral points to integral points so has a rational matrix, and thus its inverse has a rational matrix. $W_\B$ is thus rational linear since $W_C$ is, e.g. since an affine basis can be made inside $C^\circ$. Finally, since the foliation by $C^\circ$ by $W_C$ is affine and this property is preserved under pullback by $\rho$, we likewise get a foliation of $C_{\B}^\circ$ by parallel translates of $W_{\B}$.

For future use, we prove the following lemma:
\begin{lemma}\label{lemma:a-bInW_Biffrho(a-b)InW_C}
	If $a,b \in \ccone(C_\B)$ then $a - b \in W_\B$ iff $\rho(a-b) \in W_C$ 
\end{lemma}
\begin{proof}
	The forward direction is obvious, just being an application of $\rho$. For the backward direction, note $\wt(a-b) = \wt(\rho(a-b))= 0$ so $\wt(a) = \wt(b)$ so $a-b \in T C_\B$ so $a-b \in \rho^{-1}(W_C) \cap TC_\B = W_\B$
\end{proof}

Finally, we wish to define $\dep(C_\B)$. There are two candidate definitions for this object, namely as $\{x \in C_\B \mid x + W_\B$ \,\text{meets}\, $C^\circ_\B\}$, and as $\rho^{-1}(\dep(C)) \cap C_\B$. We shall show the equivalence of these definitions later.

\section{Isomorphism of Cones}

Our next goal is to prove a relationship between $\ccone(C)$ and $\ccone(C_{\B})$. We have already observed both are convex cones over polytopes and are related via an injective linear map $\rho$. We will further prove that, after passing to quotients, the two cones become isomorphic.

To begin, we define $C' \coloneqq C/W_C$ and $C_{\B}' \coloneqq C_{\B}/W_{\B}$. As a corollary to \thref{claim:quotientsCommuteWithConvexCone}, we have:

\begin{corollary}\label{cor:equivConeDefs}\phantom{.}
	\begin{itemize}[]
		\begin{minipage}{0.4\linewidth}
		\item $\ccone(C') = \ccone(C)/W_C$
		\item $\cone(C') = \cone(C)/W_C$
		\end{minipage}
		\begin{minipage}{0.4\linewidth}
		\item $\ccone(C'_{\B}) = \ccone(C_{\B})/W_{\B}$
		\item $\cone(C'_{\B}) = \cone(C_{\B})/W_{\B}$
		\end{minipage}
	\end{itemize}
\end{corollary}

With \thref{cor:equivConeDefs} in mind, we now define $\rho' \colon \ccone(C'_{\B}) \to \ccone(C')$ via $\rho'([w]) = [\rho(w)]$. It is clear $\rho'$ is linear if well-defined, and we now simultaneously prove well-definedness and injectivity. Given $[a], [b] \in \ccone(C'_{\B})$ with $a, b \in \ccone(C_{\B})$, we find:
\begin{align*}
	[a] = [b] &\iff a - b \in W_{\B}
	\\&\iff \rho(a - b) \in W_C
	\\&\iff \rho(a) - \rho(b) \in W_C
	\\&\iff [\rho a] = [\rho b]
\end{align*}

Each forward implication is immediate, and each backward implication is immediate except for the second which follows from \thref{lemma:a-bInW_Biffrho(a-b)InW_C}.

We now seek to prove $\rho'$ is surjective. To assist with the proof, we first observe Corollary 3.7 of \cite{dunfieldCountingEssentialSurfaces2022} has the following corollary:

\begin{claim}\label{claim:corr3.7Improved}
	If $F$ and $G$ are isotopic normal surfaces carried by the same lw-face $C$, then $\vec F - \vec G \in W_C$.
\end{claim}
\begin{proof}
	$F \sim G$ in particular implies $F \projsim G$ so Corollary 3.7 of \cite{dunfieldCountingEssentialSurfaces2022} implies $F^* - G^* \in W_C$. Since $F$ and $G$ are lw-surfaces and $F \sim G$ we know $\wt(F) = \wt(G)$, so scaling yields the claim.
\end{proof}

For our proof of surjectivity, we will need to be a little careful since linear maps are not generally closed maps, even in finite dimensions; indeed, consider the linear map $\pi \colon \R^2 \to \R$ given by projection to the first coordinate, and consider the image under $\pi$ of the graph of $\arctan x$. Linear maps \it{do} however map compact sets to compact sets, which we will use. Now note $\ccone(C) = \overline{\Q^{7t} \cap \cone(C)}$. Since $\Q^{7t} \cap \cone(C)$ is dense, its projection under the quotient map (which is continuous and surjective) is likewise dense. Using a result of Tollefson, we get the following:
\begin{lemma}
	The image of $\rho' \colon \ccone(C'_\B) \to \ccone(C')$ contains $\Q^{7t} \cap \cone(C)$.
\end{lemma}

\begin{proof}
	Fix $[z]$ in the projection, with representative $z \in \Q^{7t} \cap \cone(C)$.

	First we handle the integral case. Let $z \in \Z^{7t} \cap \cone(C)$, so $2z = \vec F$ where $F$ is an orientable and normal surface. Note $F$ must be an lw-surface, being carried by the lw-face $C$.  By Lemma 3.1 in \cite{tollefsonIsotopyClassesIncompressible1995}, there exists a normal surface $G$ isotopic to $F$ with $G$ carried by $\B_S$ (and thus by $C$) with weight system $\vec g$. By \thref{claim:corr3.7Improved} we have $\vec F - \vec G \in W_C$, so $[\vec F] = [\vec G]$. Recall the definition of $\rho$ yields $\rho(\vec g) = \vec G$, so
	\[\rho'([\vec g]) = [\rho(\vec g)] = [\vec G] = [\vec F] = 2z\]
	which implies $\rho'([\tfrac{1}{2} \vec g]) = z$.

	Now suppose $z$ has rational coordinates. Choose $k$ such that $kz$ has integer coordinates so is a normal surface. By the previous case we can find $w \in \ccone(C_{\B})$ such that $\rho'([w]) = [k z]$, so $\rho'([k^{-1}w]) = [z]$.
\end{proof}

Now note weight descends down to $\R^s / W_\B$ via $\wt([w]) = \wt(w)$; indeed, if $[w_1] = [w_2]$ then $w_1 - w_2 \in W_\B$, but $W_\B$ has weight zero. The same argument shows weight descends to $\R^{7t} / W_C$. We can then check that $\rho'$ preserves weight, since $\wt(\rho'([w])) = \wt([\rho(w)]) = \wt(\rho(w)) = \wt(w) = \wt([w])$. In particular, this means $\rho'$ sends $C'_\B$ into $C'$. Since $\Q^{7t} \cap C$ is dense in $C$ (here we use the fact it has weight $1$, since an arbitrary horizontal slice might contain no rational points) we likewise conclude  $\Q^{7t} \cap C'$ is dense in $C'$, and so $\rho'(C'_\B)$ is dense in $C'$ via the argument in the previous lemma; indeed, $C'$ and $C'_\B$ are the weight-1 slices in their respective quotient cones, and if $\wt([z]) = 1$ then $\wt([k^{-1}w]) = \wt(k^{-1}w) = \wt(\rho(k^{-1}w)) = \wt(z) = \wt([z]) = 1$. Since $\rho'$ is continuous and $C'_\B$ is compact, we know $\rho'(C'_\B)$ is compact, so that $\rho'(C'_\B) = C'$. By scaling, we get surjectivity at the level of closed cones.

By extending $\rho'$ to a linear map on the affine span of $\ccone(C'_\B)$ we of course have an inverse, since injectivity on an open subset of the affine span implies injectivity on the whole affine span. We conclude:

\begin{thm}
	$\rho' \colon \ccone(C'_\B) \to \ccone(C')$ is a linear homeomorphism.
\end{thm}

In particular, $\rho'$ must send interiors to interiors, so sends $\cone(C'_\B)$ to $\cone(C')$. At the level of dimensions, this has an important corollary we will need later:

\begin{corollary}
	$\dim \cone(C'_\B) = \dim \cone(C')$.
\end{corollary}

Also, since $\rho'$ preserves cone `height' (in the sense of the weight functional), we have:

\begin{corollary} \label{corr:C'IsomorphicToC'_B}
	$\rho' \colon C'_\B \to C'$ is a linear homeomorphism.
\end{corollary}

\section{Properties of Pullback Constructions}

Let us now define $\dep(C_\B) \coloneqq \{x \in C_\B \mid x + W_\B$ \,\text{meets}\, $C^\circ_\B\}$. An important observation for the following is that each class in the interior of $C / W_C$ can be geometrically thought of as a translated copy of $W_C$ intersecting $C^\circ$; in particular, each interior class of $C/W_C$ is represented by some element of $C^\circ$, and likewise each interior class of $C_\B / W_\B$ is represented by some element of $C^\circ_\B$.

In the following proof, we will write $[\cdot]_{W_C}$ and $[\cdot]_{W_\B}$ to specify the equivalence classes in $C/W_C$ and $C_\B / W_\B$ respectively. Note $x \in C_\B$ lies in $\dep(C_\B)$ iff $[x] = [c']$ for some $c' \in C_\B^\circ$, and likewise $y \in C$ lies in $\dep(C)$ iff $[x] = [c]$ for some $c \in C^\circ$.

\begin{lemma}
	$\dep(C_\B) = \rho^{-1}(\dep(C)) \cap C_\B$.
\end{lemma}
\begin{proof}
	First suppose $x \in \dep(C_\B)$, so $[x]_{W_\B} = [c']_{W_\B}$ for some $c' \in C_\B^\circ$. Since $C'_\B \cong C'$ via $\rho'$, we know there exists some $c \in C^\circ$ such that $\rho'([c']_{W_\B}) = [c]_{W_C}$. As such, 
	\[[\rho(x)]_{W_C} = \rho'([x]_{W_\B}) = \rho'([c']_{W_\B}) = [c]_{W_C}\]
	so that $\rho(x) + W_C = c + W_C$, i.e. $c \in \rho(x) + W_C$.\\Since $\rho(x) \in C$ and $c \in C^\circ$, we conclude $\rho(x) \in \dep(C)$. Since $\dep(C_\B) \subset C_\B$, we have proved $\dep(C_\B) \subseteq \rho^{-1}(\dep(C)) \cap C_\B$.

	The other direction is nearly identical, just with more details to check. Let $x \in \rho^{-1}(\dep(C)) \cap C_\B$, so $\rho(x) \in \dep(C)$ so $[\rho(x)]_{W_C} = [c]_{W_C}$ for some $c \in C^\circ$. Since $C'_\B \cong C'$ via $\rho'$, we know there exists some $c' \in C_\B^\circ$ such that $\rho'([c']_{W_\B}) = [c]_{W_C}$. As such,
	\[[\rho(c')]_{W_C} = \rho'([c']_{W_\B}) = [c]_{W_C} = [\rho(x)]_{W_C}\]
	so that $\rho(c') + W_C = \rho(x) + W_C$, meaning $\rho(c') \in \rho(x) + W_C$ and thus $\rho(c' - x) \in W_C$. Note $c' - x \in TC_\B$, so $c' - x \in \rho^{-1}(W_C) \cap TC_\B = W_\B$, i.e. $c' \in x + W_\B$ so $x \in \dep(C_\B)$.
\end{proof}

It is worth noting this result has a geometric interpretation: we can identify $C_\B$ with its image under $\rho$ (via injectivity), and then this result says the foliation $\mathcal F_\B$ of $C_\B^\circ$ by $W_\B$ is exactly the intersection with $C_\B^\circ$ of the foliation $\mathcal F_C$ of $C^\circ$ by $W_C$. Under this interpretation, collapsing $W_C$ automatically collapses $W_\B$, and thus \thref{corr:C'IsomorphicToC'_B} tells us something further of the geometry of $C_\B$. We claim \it{every} fiber $\mathcal F_\B$ intersects $C_\B^\circ$. If not, then some fiber must intersect part of the boundary of $C_\B$ which is not in $\dep(C_\B)$, but the fiber itself is in $\dep(C)$, contradicting the prior lemma. It thus makes sense to define a linear projection map $\pi_\B$ from $C$ onto $C_\B$ given by collapsing transverse directions of $W_C$, with $\pi_\B^{-1}(c') \cap C^\circ \neq \emptyset$ for every $c' \in C_\B^\circ$. Additionally, since $W_C = W_\B \oplus W'$ for some subspace $W'$ of $W_C$ which $\pi_\B$ exactly collapses to zero, with both $W_C$ and $W_\B$ rational linear, we conclude $W'$ is rational linear. This means that for any rational $c'$ carried by $C_\B^\circ$, we can find a rational $c$ carried by $C^\circ$ such that $\pi_\B(c) = c'$.

We next need to prove an analogue of Theorem 3.5 of \cite{dunfieldCountingEssentialSurfaces2022} for $C_\B$. 

\begin{lemma} \label{lemma:W_BProperties}
	Let $C$ be an lw-face, and let $F$ and $G$ be normal surfaces carried by $\dep(C_\B)$.
	\begin{enumerate}[itemsep=-0.1cm]
		\item $W_\B$ is rational linear
		\item $F \projsim G$ iff $f^* - g^* \in W_\B$
		\item If $F$ and $G$ are orientable, then $F \sim G$ iff $\vec f - \vec g \in W_\B$
		\item $W_\B \subseteq \ker(\wt)$
		\item If $F$ is carried by $C^\circ_\B$, then there exist surfaces $F_1, \ldots, F_n$ each projectively isotopic to $F$ and carried by $C^\circ_\B$ such that $f^* - f_i^*$ span $W_\B$.
	\end{enumerate}
\end{lemma}
\begin{proof}
	Most of these statements rather immediately follow from Theorem 3.5 of \cite{dunfieldCountingEssentialSurfaces2022}. That $W_\B$ is rational linear was already proved earlier. The next two conclusions have identical calculations:
	$f^* - g^* \in W_B$ iff $\rho(f^* - g^*) \in W_C$ iff $\rho(f^*) - \rho(g^*) \in W_C$ iff $F^* - G^* \in W_C$ iff $F \projsim G$, and in the orientable case $\vec f - \vec g \in W_B$ iff $\rho(\vec f - \vec g) \in W_C$ iff $\rho(\vec f) - \rho(\vec g) \in W_C$ iff $\vec F - \vec G \in W_C$ iff $F \sim G$. Note both of these bi-implication chains require an application of \thref{lemma:a-bInW_Biffrho(a-b)InW_C}. That $W_\B \subseteq \ker(\wt)$ follows from $W_C \subseteq \ker(\wt)$ and $\rho$ preserving weight. 

	The last statement is essentially taking the analogous statement for $W_C$ and pushing everything down via $\pi_\B$. Explicitly, if $F$ is carried with vector $\vec f$ by $C^\circ_\B$, then it is also carried by $\rho(C^\circ_\B)$ with vector $\vec F$, and we can choose some rational $c$ carried by $C^\circ$ with $\pi_\B(c) = F^*$. Scale $c$ to an integral $\vec G$ representing a surface $G$. By Theorem 3.5 of \cite{dunfieldCountingEssentialSurfaces2022} there exist surfaces $G_1, \ldots, G_n$ carried by $C^\circ$ such that $G \projsim G_i$ and $G^* - G^*_i$ span $W_C$. Then $\pi_\B(G^* - G^*_i) = \pi_\B(G^*) - \pi_\B(G^*_i)$ span $\pi_\B(W_C) = \rho(W_\B)$. We then note $\pi_\B(G^*) = \pi_\B(c) = F^*$, define $F_i$ to be an integral scaling of $\pi_\B(G^*_i)$, and note that since $F_i$ is carried by $\rho(C_\B^\circ)$, there exists a unique $\vec f_i  \in C_\B^\circ$ such that $\rho(\vec f_i) = \vec F_i$. We conclude $F^* - F_i^*$ span $\rho(W_\B)$, so $f^* - f_i^*$ span $W_\B$.
\end{proof}

As in normal solution space, we can allow $F$ and $G$ to wander outside $\dep C$ and still have the difference stay in $W_C$:

\begin{claim}\label{claim:corr3.7ImprovedC_B}
	If $F$ and $G$ are isotopic lw-surfaces carried by $C_\B$, then $\vec f - \vec g \in W_\B$.
\end{claim}
\begin{proof}
	Note $F$ and $G$ are also carried by $C$, with $\vec F = \rho(\vec f)$ and $\vec G = \rho(\vec g)$. By \thref{claim:corr3.7Improved} we know $\vec F - \vec G \in W_C$, so by \thref{lemma:a-bInW_Biffrho(a-b)InW_C} we conclude $\vec f - \vec g \in W_\B$.
\end{proof}
\section{Quotient Cones Exactly Kill Isotopies}

We now relate $W_C$ and $W_\B$ to isotopy relations in the cones $\ccone(C)$ and $\ccone(C_\B)$. Given two integral vectors $\vec F$ and $\vec G$ representing isotopic lw-surfaces $F$ and $G$, we say there is an \it{isotopy relation} between $\vec F$ and $\vec G$. We can \it{collapse} an isotopy relation by considering $\ccone(C) / \text{span}\{\vec F - \vec G\}$; in other words, we collapse the line between the vectors representing isotopic surfaces. We can likewise collapse multiple isotopy relations simultaneously. The same ideas all hold for $C_\B$, where we instead identify weights $\vec f$ and $\vec g$. Since $\ccone(C)$ and $\ccone(C_\B)$ have finite dimensions, by repeatedly collapsing along isotopy relations we must at some point end up with a cone without isotopy relations. When considering a cone collapsed along all its isotopy relations, we will notate it as the cone quotiented by $\sim$.

Note collapsing $\ccone(C)$ along all isotopy relations is the same as collapsing $\ccone(C)$ along all isotopy relations between orientable surfaces: indeed, if $F \sim G$ with one of $F, G$ non-orientable, we can instead consider $2F$ and $2G$, which are isotopic orientable surfaces. Quotienting along the isotopy relation between $2F$ and $2G$ identifies $\vec F$ and $\vec G$. We can likewise double weights in $\ccone(C_\B)$ to achieve orientability so that collapsing orientable isotopy relations collapses all isotopy relations. Via these observations we are able to bypass---at least in our current context---some issues with orientability that arise in \cite{dunfieldCountingEssentialSurfaces2022}.

All of the aforementioned quotients can also be carried out on subsets of $\ccone(C)$ and $\ccone(C_\B)$; of particular relevance are $\cone(C)$ and $\ccone(\dep(C))$, as well as their $C_\B$ analogues.

\begin{lemma}\label{lemma:equivDefsCollapsedConeC}
	$\cone(C') = \cone(C)/\!\!\sim$
\end{lemma}
\begin{proof}
	Let $W_C'$ denote the span of the one-dimensional spaces we quotient $\cone(C)$ by to collapse all isotopy relations, i.e. $\cone(C)/\!\!\sim \;= \cone(C)/W_C'$.

	By \thref{cor:equivConeDefs} it suffices to show $W'_C = W_C$. Let $V \subseteq W'_C$ be one of the  one-dimensional subspaces used to generate it. By definition of $W_C'$, there exist distinct, orientable, normal surfaces $F$ and $G$ carried by $C^\circ$ such that $F \sim G$ and $\vec F - \vec G \in V$. Since $F$ and $G$ are distinct we know $\vec F - \vec G$ generates $V$, but since $F \sim G$ with $F$ and $G$ being orientable and normal surfaces carried by $C$, \thref{claim:corr3.7Improved} implies $\vec F - \vec G \in W_C$. Thus $V \subseteq W_C$, so $W'_C \subseteq W_C$.

	Conversely, recall by Corollary 3.7 of \cite{dunfieldCountingEssentialSurfaces2022} that $W_C$ is spanned by $\vec G - \vec H$, where $H$ is carried by $C$, $F$ is carried by $C^\circ$, and $G$ is a component of $F$ with $G \sim H$. Let $V \subseteq W_C$ be a one-dimensional subspace generated by such a $\vec G - \vec H$. The problem we have is that $G$ or $H$ could lie on the boundary of $C$, so $V$ is not obviously generated by a difference of elements of $C^\circ$. To show this, let $\vec F' \coloneqq \vec F + \tfrac{1}{n} (\vec G - \vec H)$, where $n \in \N$ is chosen large enough so $F' \in C^\circ$. Since $G \sim H$ we have $\vec G - \vec H \in \ker(\wt)$, so $\wt(F') = \wt(F)$. Then $\vec{F''} \coloneqq \vec{2nF'} = 2nF + 2(\vec G - \vec H)$ represents an orientable, normal surface $F''$ carried by $C^\circ$, and $\vec{F''} - \vec{2n F} = 2(\vec G - \vec H)$ generates $V$.

	As an alternate proof of the converse that will be useful shortly: recall Theorem 3.5 of \cite{dunfieldCountingEssentialSurfaces2022} shows $W_C$ is spanned by $F_1^* - F_i^*$  for some surfaces $F_1, \ldots, F_n$ each carried by $C^\circ$ with $F_i \projsim F_1$. By scaling each $F_i$ simultaneously (and doubling as necessary to ensure everything is orientable), we get orientable surfaces $G_1, \ldots, G_n$ carried by $C^\circ$ such that $G_i^* = F_i^*$, $G_1^* - G_i^*$ span $W_C$, and $G_i \sim G_1$. Since $G_i$ and $G_1$ are isotopic they have the same weight, so in fact $\vec G_1 - \vec G_i$ span $W_C$. Each of these is an isotopy relation, so $W_C \subseteq W_C'$.
\end{proof}

One may wonder whether the same holds for $\ccone(\dep(C))$. In fact, we get nothing new: $\dep(C)/W_C \cong C^\circ / W_C$ is geometrically clear by definition of $\dep(C)$, and by writing $\dep(C)$ as a union of closed and open faces and applying \thref{claim:quotientsCommuteWithConvexCone} to the cone over each piece and gluing, we have $\ccone(\dep(C))/W_C = \ccone(\dep(C)/W_C)$, whence
\[\ccone(\dep(C))/W_C = \ccone(\dep(C)/W_C) \cong \ccone(C^\circ / W_C) = \ccone(C^\circ)/W_C,\] so the only difference between this and $\cone(C')$ is whether we include the origin or not (which is not isotopic to anything else, so is clearly inconsequential for current purposes).

The following proof is identical to the prior one with certain references to \cite{dunfieldCountingEssentialSurfaces2022} replaced by references to analogous statements for $C_\B$ proven prior in this paper, and the reader can safely skip to the next section.

\begin{lemma}\label{lemma:equivDefsCollapsedConeC_B}
	$\cone(C'_\B) = \cone(C_\B)/\!\!\sim$
\end{lemma}
\begin{proof}
	Let $W_\B'$ denote the span of the one-dimensional spaces we quotient $\cone(C_\B)$ by to collapse all isotopy relations, i.e. $\cone(C_\B)/\!\!\sim \;= \cone(C_\B)/W_\B'$.

	By \thref{cor:equivConeDefs} it suffices to show $W'_\B = W_\B$. Let $V \subseteq W'_\B$ be one of the  one-dimensional subspaces used to generate it. By definition of $W_\B'$, there exist distinct, orientable, normal surfaces $F$ and $G$ carried by $C_\B^\circ$ such that $F \sim G$ and $\vec f - \vec g \in V$. Since $F$ and $G$ are distinct we know $\vec f - \vec g$ generates $V$, but since $F \sim G$ with $F$ and $G$ being orientable and normal surfaces carried by $C_\B$, \thref{claim:corr3.7ImprovedC_B} implies $\vec f - \vec g \in W_\B$. Thus $V \subseteq W_\B$, so $W'_\B \subseteq W_\B$.

	Conversely, recall that by \thref{lemma:W_BProperties} we know $W_\B$ is spanned by $f_1^* - f_i^*$  for some surfaces $F_1, \ldots, F_n$ each carried by $C^\circ_\B$ with $F_i \projsim F_1$. By scaling each $F_i$ simultaneously (and doubling as necessary to ensure everything is orientable), we get orientable surfaces $G_1, \ldots, G_n$ carried by $C^\circ_\B$ such that $G_i^* = F_i^*$, $G_1^* - G_i^*$ span $W_\B$, and $G_i \sim G_1$. Since $G_i$ and $G_1$ are isotopic they have the same weight, so in fact $\vec G_1 - \vec G_i$ span $W_\B$. Each of these is an isotopy relation, so $W_\B \subseteq W_\B'$.
\end{proof}

\section[Connecting Cones to ML]{Connecting Cones to $\ML$}

In 1988, Oertel proved the following result first connecting cones to $\ML$:

\begin{reference}[(Theorem 3.5, Lemma 3.6, and Theorem 1.7 in \cite{oertelMeasuredLaminations3Manifolds1988})]
	Let $\B_1, \ldots, \B_n$ be finitely many TIBs without isotopy relations carrying all essential surfaces in $M$. Let $\mathcal H$ be the set of nontrivial homotopy classes of closed curves in $M$.
	\begin{enumerate}[itemsep=0cm]
		\item The maps $\phi_i \colon \cone(\B_i) \to \R^{\mathcal H}$ whose coordinates are length functions admit unique continuous extensions $\widebar \phi_i \colon \ccone(\B_i) \to \R^{\mathcal H}$.
		\item The maps $\widebar \phi_i$ are topological embeddings.
		\item $\ML$, viewed as a subset of $\R^{\mathcal H}$ with the subspace topology, is contained in the union of the images of the $\widebar \phi_i$, and contains the interiors of these images.
	\end{enumerate}
\end{reference}

As a corollary to (3),

\begin{corollary} \label{corr:OertelMLWork}

	$\dim \ML = \max_i \dim \image(\phi_i) = \max_i \dim \cone(\B_i) = \dim \cone(\B)$\\for some $\B \in \{\B_1, \ldots, \B_n\}$.
	
\end{corollary} 

We now prove a general result about slicing polytopes that we will utilize as a density result for a complete lw-face. Recall a face $B$ of $C$ is said to be \it{filling} (with respect to some affine space $W$) if $TC = TB + W$.

\begin{lemma}\label{lem:parallelSlicesOfConvex}
	Let $C$ be a convex polytope in $\aff(C) = \R^n$, and $W$ be a subspace of $\R^n$. Then a face $B$ of $C$ is filling iff the space of translates of $W$ intersecting $B$ is of full possible dimension, namely $\dim(\R^n / W) = \reldim C$.
\end{lemma}
\begin{proof}
	Let $\pi \colon \R^n \to \R^n / W$ be the quotient map. First note the set of translates $x + W$ intersecting $B$ (for $x \in \R^n$) is exactly $\pi(B)$; indeed, if $x + W$ meets $B$ then there exists $b \in (x+W)\cap B$ so $\pi(x) = \pi(b) \in \pi(B)$. Conversely, if $\pi(x) \in \pi(B)$ then some $b \in B$ satisfies $\pi(x) = \pi(b)$, so $b - x \in W$ so $b \in x + W$ so $(x+W)\cap B \neq \emptyset$. Also note $\pi(B)$ is a convex polytope since $B$ is. We have \[\dim \pi(B) = \dim TB - \dim (TB \cap W) = \dim (TB + W) - \dim W\] (via inclusion-exclusion), so $\dim \pi(B) = \reldim C$ iff $\dim (TB + W) = \dim C$, i.e.  $\dim \pi(B) = \reldim C$ iff $TB + W = TC$.
\end{proof}

Note the faces of codimension $\le 1$ whose interiors are in $\dep C$ must automatically be filling, since $W$ provides at least one transverse direction toward $C^\circ$. The same need not be true of codimension $2$ faces: indeed, if $C$ is a square and $W$ is a line not parallel to a side of $C$, then some translate of $W$ passes through a vertex $v$ of $C$ in $\dep(C)$, and this vertex is not filling since $\dim(Tv+W) = 1 < 2 = \dim(C)$. However, there are at most two such translates of $W$ which intersect $C^\circ$ and are problematic: every other translate intersects $\del C$ only at the interior of a filling boundary edge. It is worth noting it is not sufficient to have $\dim B + \dim W \ge \dim C$ for $B$ to be filling, as can be seen by taking the product of everything in the prior example with a closed interval. Also note that this implies that almost every translate of $W$ intersecting $C$ does so only at filling faces: the set of translates of $W$ intersecting $C$ at non-filling faces is a union of proper polyhedra, so the complement is open and dense in the space of translates intersecting $C$ (indeed, the above proof essentially shows intersection with a filling face is an open condition). In particular, we get the following:

\begin{corollary}
	Almost every translate of $W$ intersecting $C$ does so only at filling faces. Almost every rational point in $C$ lies on some translate of $W$ intersecting only filling faces.
\end{corollary}

The latter claim in the corollary follows from the density of rationals along with the openness of the intersecting condition. In particular, there must exist some rational point $x$ in $C^\circ$ such that $x + W$ intersects only filling faces.

To prepare for using this, let $S$ be some lw-surface carried positively by a lw-face $C$. Recounting the techniques of Oertel in his proof of Theorem 4 in \cite{oertelIncompressibleBranchedSurfaces1984}, we use the fact $S$ is least-weight to recursively remove product regions. Indeed, if $S$ is carried by $\B_S$ in such a way that $\B_S$ admits a trivial $I$-bundle region, then we can isotope all sheets of $S_1 \coloneqq S$ to one side while preserving being least weight (otherwise all sheets of $S$ would already be pushed to the side of less weight). This yields an isotopic $S_2$, and we note $\B_{S_2}$ is isotopic to the sub-branched surface of $\B_S$ we get from deleting the now unused sectors. The only time this doesn't simplify things is when $S$ is a union of fibers in a global fibration over $S^1$, in which case fibers on the ``top'' and ``bottom'' can be interchanged by monodromy; we simply ignore such product regions. We repeat this process with $S_2$ in place of $S_1$, and iterate. Either via Oertel's argument or by noting this process recursively pushes $S$ into open faces of $\dep(C)$ of increasingly small dimension, we conclude this process terminates in finite time to yield some $S_n$ such that $S \sim S_n$ and $\B_{S_n}$ is a TIB without isotopy relations (by \cite{oertelMeasuredLaminations3Manifolds1988}) carrying $S_n$---and thus $S$---positively up to isotopy.

Importantly, at each step $k$ of this process, $S_k \sim S$, so $S_k$ and $S$ must lie on the same translate of $W_C$ (and $S_k$ must be in $\dep(C)$). By our prior lemma, we can always find a rational point in $C^\circ$---and thus a surface $S$ carried positively by $C$, by rescaling---such that $S + W_C$ intersects $C$ only along compatible faces. We conclude $S_n$, wherever it lands, does so in a compatible lw-face of $C$. 

As a consequence of this and the prior results on cone dimensions, note that for every complete face $C$, there exists a surface $S$ carried positively by $C$ such that $\B = \B_{S_n}$ is without isotopy relations and $S_n$ is positively carried by a compatible face $B \subseteq C$. This means $W_\B = 0$, so $\cone(\B) \cong \cone(B)/W_B$, i.e. $\dim \cone(\B) = \dim(\cone(B)/W_B) = \reldim B + 1 = \reldim C + 1$. This is the largest possible dimension of a cone we could get from a surface carried by $\dep(C)$; indeed, \it{any} such $S'$ is isotopic to some $S'_n$ inside the interior of subface $D$ of $C$ after running Oertel's construction, and $\dim \cone(\B_{S'}) = \reldim D + 1 \le \reldim C + 1$ by \thref{cor:reldimMaxOnComplete}. Since each lw-surface is carried by $\dep C$ for some unique complete face $C$ and the branched surfaces generated by such surfaces are exactly Oertel's TIBs, we conclude the TIB of maximal dimension must in fact be the relative dimension of the cone over some complete face. By \thref{corr:OertelMLWork} this mean $\dim \ML = \dim C - \dim W_C + 1$ for some complete face $C$. 

Recall a lw-face $C$ is said to be \it{essential} if it carries only essential surfaces. Note we have nowhere chosen a triangulation for any of the prior proofs, and everything works equally well in the case of closed surfaces for both material and ideal triangulations. Via Lemma 4.10 of \cite{dunfieldCountingEssentialSurfaces2022}, we can choose $\tau$ to be a $\del$-efficient ideal triangulation to force a face to be essential iff it carries no $\del$-parallel normal surfaces; equivalently, a face is essential iff it carries a \it{single} essential surface. As remarked in Proposition 4.7 of \cite{oertelMeasuredLaminations3Manifolds1988}, Oertel's TIBs carry no $\partial$-parallel surfaces. As a consequence, given some complete face $C$ and an lw-surface $S$ positively carried, the normal surface $S_n$ we get must be essential since it is normally carried by the TIB $\B_{S_n}$, so that $C$ is essential. As a consequence to this and the prior paragraph,

\begin{thm}\label{thm:mainResult}
Let $M$ be a compact orientable irreducible $\del$-irreducible 3-manifold. Then:
	{
	\begin{enumerate}[itemsep=0cm,leftmargin=2em,font=\upshape]
		\item If $\tau$ is a material triangulation of $M$, then $\dim( \ML)$ is the \\maximum of $\dim C - \dim W_C + 1$ among complete lw-faces $C$.
		\item If $\tau$ is an ideal triangulation of $M$, then $\dim( \ML_0)$ is the \\maximum of $\dim C - \dim W_C + 1$ among essential lw-faces $C$.
	\end{enumerate}}
\end{thm}

This resolves Conjecture 1.12 of \cite{dunfieldCountingEssentialSurfaces2022} as stated in the introduction, showing the number of closed, orientable surfaces of Euler characteristic at most $-2n$ in a 3-manifold as stated in the conjecture is asymptotic to $c n^{\ML_0}$. Note the authors further conjectured a statement on the constant $c$ in terms of an analogue of the Thurston measure, and this remains open.

It is natural to suppose such a statement should also hold for surfaces which are non-orientable or have boundary; in the former case one runs into issues with not all non-orientable surfaces being carried by least weight faces, and even restricting to orientable surfaces in the latter case requires care. 
\pagebreak
\section{Appendix}
\begin{claim}\label{claim:injectivityResult}
 Suppose $f \colon \R^n \to \R^m$ is linear, $\ker f$ is a rational linear subspace, $A \subseteq \R^n$, $A$ is open, and $f$ is injective on the rational points in $A$. Then $f$ is injective on $A$.
\end{claim}
\begin{proof}
	If $\ker f = 0$ then we are done, since $f$ is already injective. Otherwise, since $\ker f$ is rational linear, there exists some nonzero rational $v \in \ker f$. Since $A$ is open, it contains some rational point $x$. Choose a rational sequence $t_n \to 0$. Then $x + t_n v \in A$ for sufficiently large $n$ since $A$ is open, and also $x + t_n v$ is rational, but $f(x+t_n v) = f(x) + t_n f(v) = f(x)$ with $x \neq x + t_n v$ (since both $t_n$ and $v$ are nonzero), contradicting injectivity on rational points in $A$. 
\end{proof}

\begin{claim}\label{claim:quotientsCommuteWithConvexCone}
	Let $W$ be a linear subspace of $\R^n$, and let $\pi \colon \R^n \to \R^n/W$ be the quotient map. Let $A \subset \R^n$ be a convex set. Then $\pi(\ccone(A)) = \ccone(\pi(A))$ and $\pi(\cone(A)) = \cone(\pi(A))$.
\end{claim}
\begin{proof}
	The first equality is a straightforward calculation:
	\begin{align*}
	\pi(\ccone(A)) &= \{\pi(\lambda a) \mid \lambda \in \R_{\ge 0}, a \in A \}
	\\&=\{\lambda \pi(a) \mid \lambda \in \R_{\ge 0}, a \in A \}
	\\&= \ccone(\pi(A))
	\end{align*}

	The second equality could either be proved by showing cones and relative interior can be interchanged (see Theorem 6.9 of \cite{rockafellarConvexAnalysis2015}), or can be proved similarly to the prior calculation by recalling affine maps preserve relative interior:
	\begin{align*}
	\pi(\cone(A)) &= \pi \left( \textstyle{\bigcup_{\lambda > 0}} \lambda \relint(A) \right)
	\\&=  \textstyle{\bigcup_{\lambda > 0}}  \pi \left(\lambda \relint(A) \right)
	\\&=  \textstyle{\bigcup_{\lambda > 0}}  \lambda \pi \left( \relint(A) \right)
	\\&=  \textstyle{\bigcup_{\lambda > 0}}  \lambda \relint \left( \pi(A) \right)
	\\&=\cone(\pi(A))
	\end{align*}
\end{proof}

\pagebreak
\bibliography{references}
\end{document}

\typeout{get arXiv to do 4 passes: Label(s) may have changed. Rerun}